\documentclass{article}
\usepackage{amsmath,amssymb,amsthm}
\usepackage{graphics}

\newtheorem{theorem}{Theorem}[section]

\newtheorem{prop}{Proposition}[section]
\newtheorem{remark}{Remark}[section]

\newcommand{\IND}{\mathop{\rm ind}\nolimits}
\newtheorem{cor}{Corollary}[section]

\title{Curvature extrema and four-vertex theorems for polygons and polyhedra}

\author{Oleg R. Musin}

\begin{document}
\date{}
\maketitle

\begin{abstract} Discrete analogs of extrema of curvature and generalizations of the four-vertex theorem to the case of polygons and polyhedra are suggested and developed. For smooth curves and polygonal lines in the plane, a formula relating the number of extrema of curvature to the winding numbers of the curves (polygonal lines) and their evolutes is obtained. Also are considered higher-dimensional analogs of the four-vertex theorem for regular and shellable triangulations.
\end{abstract}

\section{Introduction}

\noindent{\bf 1.1.} The notion of curvature always played a special
role in mathematics and theoretical physics. Despite the fact that
this notion lives in science for more than two hundred years, the
flow of papers devoted to it does not draw, and the area of
applications widens.

In addition to the mathematical aspect, the subject of the present
paper has an entirely practical aspect. The distribution of the
curvature function and its singular points are of interest in many
applied areas. The present author participated in projects related
to cartographic generalization (generalization of the image of an
object on a map under passage to smaller scale), when it is
necessary to distinguish singular points on a curve or a surface:
usually, they are zeros and (local) extrema of curvature \cite{6,
12, 19}. Also, extrema of curvature and their distribution are
important characteristics in computer modeling of surfaces.

\medskip

\noindent{\bf 1.2.}  In differential geometry, the notion of curvature
admits a ``good'' definition only for the class of ($C^2$-)smooth
curves and surfaces. In the applications, one usually deals with
discrete models, where a curve is represented by a polygonal line
and a surface is represented by a polyhedron. Thus, there arises a
problem of {\it correct\/} definition of the notions of curvature
and extrema of curvature in the discrete case.

What does the correctness mean in this case? To our mind, the
discrete analog of curvature must satisfy a certain global property
of curvature fulfilled in the smooth case.

\medskip

\noindent{\bf 1.3.}  One of such global properties of the Gauss
curvature is the Gauss--Bonnet theorem, the discrete analog of which
is well known. Let $M$ be a polygon or a closed two-dimensional
polyhedron. We assume that the curvature of $M$ vanishes everywhere
except the vertices. In the first case, the ({\it angular\/}) {\it
curvature\/} of $M$ at a vertex $A$ is the quantity

$$
\pi - \angle A,
$$
and in the second case it is the quantity

$$
2\pi - (\text{the sum of the planar angles of the faces of $M$ at a
vertex $A$}).
$$
If $M$ is a polygon, then the sum of the curvatures at its vertices
is equal to $2\pi$, and if $M$ is a polyhedron, then the sum is
equal to $2\pi \chi(M)$, where $\chi(M)$ is the Euler characteristic
of $M$. We easily see that for such definition of curvature we
automatically obtain the same assertion as in the Gauss--Bonnet
theorem for the smooth case.

\medskip

\noindent{\bf 1.4.} To our mind, another not less important property of
curvature is the {\it four-vertex theorem\/} for plane curves. Here
it suffices to refer to V.~I.~Arnold, who said many times (and
proved his words by numerous publications, see, e.g., \cite{2,3,4,5}),
that this theorem reflects a fundamental property of dimension two.

We start our paper with a short history of the four-vertex theorem (Section 2).
 ``Discrete'' versions of this theorem for polygons are well
known in geometry: it is famous lemmas by Cauchy and
A.~D.~Aleksandrov.

\medskip

\noindent{\bf 1.5.} For smooth curves, the {\it vertices\/} are
critical points of the curvature function on the curve. In the
discrete case, where the curve is a closed polygonal line, there are
several definitions of the notion of ``vertex,'' or, more precisely,
extremal vertex, see~\cite{14, 21, 22, 25}. In these papers, under
the corresponding restrictions on the polygon, discrete versions of
the four-vertex theorem are obtained. In Section 3, a definition of an
extremal vertex of a polygonal line is suggested such that the
corresponding point of the caustic (evolute) is a cusp (a point of
return). It turns out that there is a simple relationship between
the number of extremal vertices and the winding numbers of the curve
and its caustic. A similar formula relates the number of the tangent
lines of a curve $\Gamma$ drawn from a point $P$ on the plane to the
winding number of $\Gamma$ and the degree of $\Gamma$ with respect
to $P$.

\medskip

\noindent{\bf 1.6.}  
In Section 4, we consider versions of the four-vertex theorem
for regular and shellable triangulations of a ball. 
 
Consider a $d$-dimensional simplicial polytope $P$ in
$d$-dimensional Euclidean space. We call this polytope generic if
it has no $d+2$ cospherical vertices and is not a $d$-dimensional
simplex. From now on, we consider only generic simplicial
polytopes. Each $(d-2)$-dimensional face uniquely defines a
neighboring sphere going through the vertices of two facets
sharing this $(d-2)$-dimensional face. Neighboring sphere is
called empty if it does not contain other vertices of $P$ and it
is called full if all other vertices of $P$ are inside of it. We
call an empty or full neighboring sphere extremal. Schatteman
(\cite{21}, Theorem 2, p.232) claimed the following theorem:

\medskip

\noindent{\bf Theorem [Schatteman].} {\it For each convex $d$-dimensional polytope
$P$ there are at least $d$ $(d-2)$-dimensional faces defining
empty neighboring spheres and at least $d$ $(d-2)$-dimensional
faces defining full neighboring spheres.}

\medskip

For $d=2$ this result is well known (see 2.7).

Although we do not know whether this theorem is false for $d>2$, in \cite{AGMT} we show that Schatteman's proof  has certain gaps. 
In the original version of this paper we claimed that any regular triangulation of a convex $d$-polytope has at least $d$ ``ears''. For a proof we used the same arguments as in  \cite{21}.  Thus, the $d$ - ``ears''  problem of a regular triangulation is still open.


\section {Four-vertex theorem and its generalizations}


\noindent{\bf 2.1.}  We define an {\it oval} as a convex smooth closed
plane curve. The classical four-vertex theorem by Mukhopadhayaya
\cite{13} published in 1909 says the following:

\medskip

{\it The curvature function on an oval has at least
four local extrema (vertices).}

\medskip

It is well known that any continuous function on a compact set has
at least two (local) extrema: a maximum and a minimum. It turns out
that the curvature function has at least four local extrema.

The paper was noticed, and generalizations of the result appeared
almost immediately. In 1912, A.~Kneser showed that convexity is not
a necessary condition and proved the four-vertex theorem for a
simple closed plane curve.

\medskip

\noindent{\bf 2.2.} The famous book \cite{8} by W.~Blaschke (first
published in 1916), together with other generalizations, contains
a ``relative'' version of the four-vertex theorem. Here, we preserve
the formulation and notation from \cite[p. 193]{8}.

\medskip

{\it Let $C_1$ and $C_2$ be two (positively
oriented) convex closed curves, and let $do_1$ and $do_2$ be arc
elements at points with parallel (and codirected) support lines.
Then the ratio $do_1/do_2$ has at least four extrema.}

\medskip

In the case where $C_2$ is a circle, this theorem turns into the
theorem on four vertices of an oval.

\medskip

\noindent{\bf 2.3.} In 1932, Bose~\cite{9} published a remarkable
version of the four-vertex theorem in the spirit of geometry ``in
the large.'' While in the classical four-vertex theorem the extrema
are defined ``locally,'' here they are defined only ``globally.''

Let $G$ be an oval no four points of which lie on a circle. We
denote by $s_-$ and $s_+$ ($\text{resp}$., $t_-$ and  $t_+$) the
number of its circles of curvature ($\text{resp}$., the circles
touching $G$ at exactly three points) lying inside ($-$) and outside
($+$) the oval $G$, respectively. (The {\it curvature circle\/} of
$G$ at a point $p$ touches $G$ at $p$ and has radius $1/k_G(p)$,
where $k_G(p)$ is the curvature of $G$ at $p$.)

In this notation, we have the relation

$$
s_- - t_- = s_+ - t_+ = 2.
$$

If we define {\it vertices\/} as the points of tangency of the oval
$G$ with its circles of curvature lying entirely inside or outside
$G$, then these formulas imply that the oval $G$ has at least four
vertices. It is worth mentioning that this fact was proved by
H.~Kneser for 10 years before Bose. (Actually, H.~Kneser is a son of A.~Kneser, and so two interesting
strengthenings of the four-vertex theorem belong to one family.)

\medskip

\noindent{\bf 2.4.} Publications related to the four-vertex theorem did
not stop from that time, and their number considerably increased in
the recent years (see~\cite{2,3,4,5, 18}, etc.), to a large extent
owing to papers and talks by V.~I.~Arnold. In the above papers,
various versions of the four-vertex theorem for plane curves and
convex curves in ${\Bbb R}^d$, and their singular points (vertices)
are considered: critical points of the curvature function,
flattening points, inflection points, zeros of higher derivatives,
etc. In \cite{23} there is a long list of papers
devoted to this direction.

\medskip

\noindent{\bf 2.5.} It is of interest that the first discrete analog of
the four-vertex theorem arose for almost 100 years before its smooth
version. (I use this opportunity to thank V. A. Zalgaller, who
brought this fact to my attention.) In 1813, Cauchy, in his splendid
paper on rigidity of convex polyhedra, used the following lemma:

\medskip

\noindent{\bf Cauchy lemma.} {\it Let $M_1$ and $M_2$ be convex $n$-gons with
sides and angles $a_i$, $\alpha_i$ and $b_i$, $\beta_i$,
respectively. Assume that $a_i = b_i$ for all $i=1,\dots,n$. Then
either $\alpha_i = \beta_i$, or the quantities $\alpha_i - \beta_i$
change sign for $i=1,\dots,n$  at least four times.}

\medskip

In Aleksandrov's book \cite{1}, the proof of the uniqueness of a
convex polyhedron with given normals and areas of faces involves a
lemma, where the angles in the Cauchy lemma are replaced by the
sides. We present a version of it, which is somewhat less general
than the original one.

\medskip

\noindent{\bf A.~D.~Aleksandrov's lemma.} {\it Let $M_1$ and $M_2$ be two
convex polygons on the plane that have respectively parallel sides.
Assume that no parallel translation puts one of them inside the
other. Then when we pass along $M_1$ (as well as along $M_2$), the
difference of the lengths of the corresponding parallel sides
changes the sign at least four times.}

We easily see the resemblance between the above relative four-vertex
theorem for ovals (apparently belonging to Blaschke) with the Cauchy
and Aleksandrov lemmas. Furthermore, approximating ovals by
polygons, we easily prove the Blaschke theorem with the help of any
of these lemmas.

The Cauchy and Aleksandrov lemmas easily imply four-vertex theorems
for a polygon.

\medskip

\noindent{\bf Corollary of the Cauchy lemma.} {\it Let $M$ be an equilateral
convex polygon. Then at least two of the angles of $M$ do not exceed
the neighboring angles, and at least two of the angles of $M$ are
not less than the neighboring angles.}

\medskip

\noindent{\bf Corollary of the Aleksandrov lemma.} {\it Let all angles of a
polygon $M$ be pairwise equal. Then at least two of the sides of $M$
do not exceed their neighboring sides, and at least two of the sides
of $M$ are not less than their neighboring sides.}

\medskip

\noindent{\bf 2.6.} In the applications, the {\it curvature radius\/}
at a vertex of a polygon is usually calculated as follows. Consider
a polygon $M$ with vertices $A_1, \dots, A_n$. Each vertex $A_i$ has
two neighbors: $A_{i-1}$ and $A_{i+1}$. We define the curvature
radius of $M$ at $A_i$ as follows:

$$
R_i(M) =R(A_{i-1}A_{i}A_{i+1}).
$$

\medskip

\noindent{\bf Theorem \cite{14}.} {\it Assume that $M$ is a convex polygon and
for each vertex $A_i$ of $M$ the point $O(A_{i-1}A_{i}A_{i+1})$ lies
inside the angle $\angle{A_{i-1}A_{i}A_{i+1}}$. Then the theorem on
four local extrema holds true for the (cyclic) sequence of the
numbers $R_1(M)$, $R_2(M)$, \dots, $R_n(M)$, i.e., at least two of
the numbers do not exceed the neighboring ones, and at least two of
the numbers are not less than the neighboring ones.}

\medskip

Furthermore, this theorem generalizes the four-vertex theorems
following from the Cauchy and Aleksandrov lemmas.

\medskip

\noindent{\bf 2.7. Discrete version of H.~Kneser's four-vertex
theorem.} A circle $C$ passing through certain vertices of a polygon
$M$ is said to be {\it empty\/} (respectively, {\it full\/}) if all
the remaining vertices of $M$ lie outside (respectively, inside)
$C$. The circle $C$ is {\it extremal\/} if $C$ is empty or full.

\medskip

\noindent{\bf Theorem [folklore].} {\it Let $M=A_1\dots A_n$ be a convex $n$-gon, $n>3$,
no four vertices of which lie on one circle. Then at least two of
the $n$ circles $C_i(M):=C(A_{i-1}A_{i}A_{i+1})$, $i=1,\dots,n$, are
empty and at least two of them are full, i.e., there are at least
four extremal circles.}

\medskip

(S.~E.~Rukshin told the author that this result for many years is
included in the list of problems for training for mathematical
competitions and is well known to St.~Petersburg school students
attending mathematical circles.)

\medskip

\noindent{\bf 2.8.} It is also easy to suggest a direct generalization
of the Bose theorem for the polygons from the statement of Theorem
2.7.

\medskip

\noindent{\bf Theorem.} {\it We denote by $s_-$ and $s_+$ the numbers of empty
and full circles among the circles $C_i(M)$, and we denote by $t_-$
and $t_+$ the numbers of empty and full circles passing through
three pairwise nonneighboring  vertices of $M$, respectively. Then,
as before, we have}

$$
s_- - t_- = s_+ - t_+ = 2.
$$

The author suggested this fact as a problem for the All-Russia
mathematics competition of high-school students in 1998. (The
knowledge of Theorem~2.7 on four extremal circles, which is a
corollary of this fact, did not help much in solving this problem.)

\medskip

\noindent{\bf 2.9.} One more generalization of the Bose theorem is
given in \cite{25}, where one considers the case of an {\it
equilateral\/} polygon, which is not necessarily convex.

\medskip

\noindent{\bf 2.10.} V.~D.~Sedykh \cite{22} proved a theorem on four
support planes for strictly convex polygonal lines in ${\Bbb R}^3$,
and the main corollary of this theorem is also a version of
Theorem~2.7.

\medskip

\noindent{\bf Theorem [Sedykh]} {\it If any two neighboring vertices of a
polygon $M$ lie on an empty circle, then at least four of the
circles $C_i(M)$ are extremal.}

\medskip

It is clear that convex and equilateral polygons satisfy this
condition. Furthermore, Sedykh constructed examples of polygons
showing that his theorem is wrong without this assumption.






\section{Extremal vertices, cusps, and caustics}


\noindent{\bf 3.1.} As already noted, the four-vertex theorem holds
true for any smooth simple (i.e., without self-intersections) curve.
Generally speaking, we cannot refuse from the assumption of
simplicity of the curve. For example, the lemniscate (figure eight)
has only two vertices. We observe that the winding number of the
lemniscate vanishes. Another example: an ellipse has exactly four
vertices, and the winding number of the ellipse is equal to 1. We
see that there is a certain relationship between the winding number
of the curve and the number of vertices.

One of the approaches to the proof of the four-vertex theorem for
ovals involves the relationship between vertices, cusps, and the
winding number of the caustic of a curve. In particular, this
relationship is used in the proof in \cite{14}, as well as in a
number of assertions in papers by Arnold (see~\cite{2, 3} etc.).

We present a simple formula relating the number of vertices of a
plane curve with the winding numbers of the curve and its caustic,
which also holds true for an arbitrary closed polygonal line. For a
polygonal line, we define the notion of {\it vertex\/} (as points of
local extremum of curvature) without defining the curvature. At the
same time, this definition completely corresponds to the notion of
vertex in the smooth case.

\medskip

\noindent{\bf 3.2.} Let $\Gamma$ be a closed oriented polygonal line in
the plane with vertices $V_1$, \dots, $V_n$ ordered cyclically. A
vertex $V_i$ is said to be {\it positive\/} if the left angle at $V$
(when we pass along the polygonal line $\Gamma$ in accordance with
the orientation) is at most $\pi$, Otherwise, $V_i$ is {\it
negative\/}.

Let $C_i=C(V_{i-1}V_iV_{i+1})$. (For simplicity, we assume that no
three sequential vertices of $\Gamma$ lie on one line.)

Assume that a vertex $V_i$ is positive. We say that {\it the
curvature at the vertex $V_i$ is greater\/} (respectively, {\it
less\/}), {\it than the curvature at the vertex\/} $V_{i+1}$ and
write $V_i \succ V_{i+1}$ (respectively, $V_i \prec V_{i+1}$) if the
vertex $V_{i+1}$ is positive and $V_{i+2}$ lies outside
(respectively, inside) the circle $C_i$, or the vertex $V_{i+1}$ is
negative, and $V_{i+2}$ lies inside (respectively, outside) the
circle $C_i$.

In order to compare the curvatures at the vertices $V_i$ and
$V_{i+1}$ in the case where the vertex $V_i$ is negative, we replace
in the above definition the word ``greater'' by the word ``less,''
and the word ``outside'' by the word ``inside.'' We easily see that
this definition is correct, i.e., that if $V_i \succ V_{i+1}$, then
$V_{i+1} \prec V_i$.

\medskip

\noindent{\bf 3.3. Extremal vertices.} A vertex $V_i$ of a polygonal
line is {\it extremal\/} if the curvature at $V_i$ is greater
(respectively, less) than the curvatures at two neighboring
vertices, i.e.,

$$
V_{i+1} \prec V_i \succ  V_{i-1}\quad \text{or}\quad V_{i+1} \succ
V_i\prec  V_{i-1}.
$$

\begin{remark} In \cite{22}, a vertex $V_i$ is said to be {\it
support\/} if all vertices of the polygonal line $\Gamma$ lie to one
side from the circle $C_i$. According to this definition, it would
be natural to say that a vertex $V_i$ is an {\it extremal\/} (or a
vertex {\it of local support\/}) if both vertices $V_{i-2}$ and
$V_{i+2}$ simultaneously lie inside or outside the circle $C_i$.
However, such definition does not always agree with that presented
above. It is easy to give an example of a nonconvex quadrangle
having a vertex of local support which is not extremal.
\end{remark}

\medskip

\noindent{\bf 3.4. Caustics.} Following Arnold, we define the {\it
caustic\/} of a smooth plane curve $C$ as the set of centers of
curvature of the points of $C$. (In the classical differential
geometry, the caustic is called {\it evolute\/}.)

By analogy with the smooth case, the caustic can be defined for an
arbitrary closed polygonal line $\Gamma=V_1\dots V_n$. We set

$$
O_i=O(V_{i-1}V_iV_{i+1}).
$$
Then the polygonal line $K(\Gamma)=O_1\dots O_n$ will be called the
{\it caustic\/} of the polygonal line $\Gamma$.

In the smooth case, the center of curvature of a vertex of a curve
is a cusp of the caustic. Consider a point $V$ of a smooth curve
$\Gamma$ with nonzero curvature. Let $V',V''\in \Gamma$ be two
points close to $V$ lying on equal distance $h$ to different sides
of $V$. We observe that the point $O(V'VV'')$ tends to the center of
curvature (i.e., to a point of the caustic) as $h\to0$. We easily
see that

$$
|\angle V'VV'' - \angle O(V')O(V)O(V'')| = \begin{cases}
\pi &\text{if $V$ is a vertex},\\
0   &\text{otherwise}.
\end{cases}
$$
Thus, the absolute value of the difference between the angles at a
vertex and at a cusp is equal to $\pi$.

Accordingly, we define {\it cusps\/} of the caustic $K(\Gamma)$ for
a polygonal line.

\medskip

\noindent{\bf 3.5. Cusps.} A vertex $O_i$ of the caustic $K(\Gamma)$ is
a {\it cusp\/} if

$$
\angle V_{i-1}V_iV_{i+1} - \angle O_{i-1}O_iO_{i+1}=\pm\pi.
$$

\begin{theorem} A vertex $V_i$ of the polygonal line $\Gamma$
is extremal if and only if the vertex $O_i$ is a cusp of the caustic
$K(\Gamma)$.
\end{theorem}

This theorem shows that the definition of extremality of a vertex of
a polygonal line agrees with the corresponding definition in the
smooth case.

\begin{proof} We simply consider all cases of positiveness and
negativeness of the vertices $V_{i-1}$ and $V_{i+1}$ and position of
the vertices $V_{i-2}$ and $V_{i+2}$ with respect to the circle
$C_i$.
\end{proof}

\medskip

\noindent{\bf 3.6. Winding number of a curve.} Let $\Gamma$ be a closed
oriented smooth plane curve or a polygonal line. (It is possible
that $\Gamma$ has self-intersections.) For simplicity, we assume
that $\Gamma$ is generic, i.e., the smooth curve has no rectilinear
parts, and the polygonal line has no three sequential vertices lie
on one line.

As before, an extremal vertex $V$ of the curve $\Gamma$ is said to
be {\it positive\/}, if the center of curvature of $V$ lies to the
left with respect to the positive direction, and {\it negative\/}
otherwise. We denote by $N_+$ (respectively, $N_-$) the number of
positive (respectively, of negative) extremal vertices of the curve
$\Gamma$.

For $\Gamma$, the notion of the winding number $\IND(\Gamma)$ is
defined. For a polygonal line, we can define the winding number as

$$
\frac1{2\pi}\sum_i( \pi-\angle V_{i-1}V_iV_{i+1}),
$$
and for an oriented smooth curve -- as the winding number of the
tangent vector of the curve, or as the quantity

$$
\frac1{2\pi}\int k(s)ds,
$$
where $k(s)$ is the curvature at the point $s$.

\medskip

\noindent{\bf 3.7. The winding number of the caustic of a curve.} In
order to formulate the main result of this section, we must define
the notion of winding number also for the caustic. In this case, the
above definition of the winding number formally does not do because
for the points of the curve $\Gamma$ at which the curvature vanishes
the center of curvature lies at ``the point at the infinity,'' and
the tangent vector is not defined at cusps. However, it is easy to
see that in this case we can also correctly define the winding
number. We only observe that when we pass along the curve the
rotation of the tangent vector of the caustic is continuous
everywhere (also including the points of the curve where the
curvature vanishes), except the cusps, where a jump by $\pi$ occurs.
The winding number of the caustic can be rigorously defined as the
limit of the winding number of the caustic of the polygonal line
uniformly approximating the curve.

\begin{theorem} $N_+ - N_- = 2\IND(\Gamma) - 2\IND(K(\Gamma))$.
\end{theorem}

\begin{proof} First, we consider the case of a polygonal line. If a
vertex $V_i\in \Gamma$ is not extremal, then

$$
\angle O_{i-1}O_iO_{i+1}=\angle V_{i-1}V_iV_{i+1}.
$$
For an extremal vertex, these angles differ by $\pi$ (by $-\pi$) if
the vertex is positive (negative). Substituting these angles in the
expression for the winding number of a polygonal line, we obtain the
required formula.

Approximating a smooth curve by a polygonal line, in the limit we
obtain our formula and for a smooth curve. 
\end{proof}

How does Theorem 3.2 relate to the four-vertex theorem? We explain
this for the classical case of an oval. We observe that for an oval
we have $N_-=0$, and $\IND(\Gamma) =1$. Therefore, the four-vertex
theorem for an oval is obtained from Theorem 3.2 with the help of the
additional fact that $\IND(K(\Gamma))<0$ (for the proof see, e.g.,
\cite{14}).

V.~I.~Arnold turned my attention to the fact that the above formula
resembles the formula for the Maslov index of spherical curves in
his paper \cite{3}. It turns out that the list of formulas of such
type can be continued.

\medskip

\noindent{\bf 3.8. Problem on the number of tangent lines.} Consider
the problem on the number of the tangent lines of a curve passing
through a given point. Let $\Gamma$ be a smooth closed plane curve
or a polygonal line, and let $p$ be a point on the plane. We
consider the tangent lines drawn from $p$ to $\Gamma$. If $\Gamma$
is a polygonal line, then a {\it tangent line\/} is a line $l$,
passing through the vertex $V$ of $\Gamma$ so that the vertices of
$\Gamma$ neighboring with $V$ lie on one side of $l$. As before, we
give the sign of points of tangency in accordance with the sign of
the winding number of the curve, and we denote by $N_+$ (by $N_-$)
the number of positive (negative) points of tangency. Here,
$\IND_p(\Gamma)$ denotes the degree of the curve $\Gamma$ with
respect to $p$.

\begin{theorem} $N_+ - N_- = 2\IND(\Gamma) - 2\IND_p(\Gamma)$.
\end{theorem}

\begin{proof} It is is similar to the proof of Theorem~3.2. First, we
prove the assertion for a polygonal line, and then, by way of
passage to the limit, for a plane smooth curve. 
\end{proof}

To our mind, the obvious similarity between Theorems 3.2 and 3.3 and the
Arnold formula shows that they are special cases of a certain
general fact on the number of tangent lines.

\medskip

\noindent{\bf 3.9. Case of polyhedra.} In conclusion of this section,
we note that all definitions presented above for polygonal lines are
easily transferred to the case of polyhedra.

Let $M$ be a simplicial $d$-dimensional polyhedron immersed into
${\Bbb R}^{d+1}$. The $(d-1)$-faces of $M$ play the role of the
vertices $V_i$, and the sphere containing all the $d$ vertices of
and the two vertices neighboring to the corresponding $(d-1)$-face
plays the role of the circle $C_i$. As in the case where $d=1$, the
sign of the angle of the face ($>$ or $<$ $\pi$) and position of the
neighboring vertices (inside or outside) with respect to these
spheres determine the order $\prec$ and $\succ$ for neighboring
$(d-1)$-faces of $M$, and, therefore, allow us to define (local)
extremality. In addition to extremals, for $d>1$ there also may be
other ``singularities,'' e.g., ``saddles,'' for $d=2$. It would be
interesting to find a generalization of Theorem 3.2 for the
higher-dimensional case.


\section{Regular and shellable triangulations.}



Let $D$ be a simplicial $d+1$-polytope, i.e. $D$ is homeomorphic to a $(d+1)$-ball and all vertices of $D$ lie on the boundary. Let $t$ be a triangulation of $D$. Let us call an ``ear''  a boundary simplex of $t$. A precise definition of this picturesque term (which is adopted in computational geometry) is as follows: a $(d+1)$-simplex $s$ of 
$t$  is an ``{\it ear\/}'' of $t$ if at least two of its $d$-faces lie on the boundary of $D$. How small can be the number  of  ``ears'' of $t$? 

\medskip

First we consider the simple case where $d=1$.

\begin{prop} Let $t$ be a triangulation of an $n$-gon $M$,
where $n>3$ and all vertices of the triangulation are vertices of
$M$. Then the number of the ``ears'' of the triangulation $t$
exactly by $2$ exceeds the number of the inner triangles.
\end{prop}
\begin{proof} We denote by $x$ the number of ``ears'' in $t$, we
denote by $y$ the number of the triangles exactly one side of which
is a side of $M$, and we denote by $z$ the number of the inner
triangles. Then $M$ has $2x+y$ sides, i.e.,

$$
2x+y=n. \eqno (*)
$$
On the other hand, the total number of triangles is equal to $n-2$,
i.e.,

$$
x+y+z=n-2. \eqno (**)
$$
Subtracting {$(**)$} from {$(*)$}, we obtain $x-z=2$, as
required. 
\end{proof}

\begin{cor} Each triangulation of a  polygon $M$ has at
least two ``ears''. 
\end{cor}

This assertion can be regarded as a topological version of the
four-vertex theorem, and the formula $x-z=2$ can be regarded as a
version of the Bose theorem.

In the general case, a triangulation of the $d$-ball, where $d>2$,
may have no ``ears.'' In 1958, M.~E.~Rudin~\cite{20} constructed the
first example of such a triangulation for a convex polyhedron in
${\Bbb R}^{3}$.

Actually, she constructed a {\it nonshellable\/} triangulation of
the ball $D^3$. (A triangulation $t$ of the ball $D^d$ all vertices
of which lie on the boundary of the ball is {\it shellable\/} if the
$d$-simplices $s_j$ in the triangulation $t$ can be ordered so that
the union $P_k = s_1\cup s_2\cup\ldots\cup s_k$ is homeomorphic to
the $d$-ball for any $k=1,\dots,m$. Here, $m$ denotes the number of
 $d$-simplices of  $t$.)

\begin{prop} Any shellable triangulation of the ball $D^d$
has at least one ``ear''.
\end{prop}

\begin{proof} Consider the simplex $s_m$. Then, by definition, $P_m =
P_{m-1}\cup s_m$, where $P_m$ and $P_{m-1}$ are homeomorphic to the
ball. This means that $s_m$ is an ``ear'' because otherwise
$P_{m-1}$ is not homeomorphic to the ball. (Only the operation of
``cutting off an ear'' allows us to determine shelling a
triangulation in the inverse order.)
\end{proof}

In particular, this proposition implies that if a triangulation has
no ``ears,'' then it is not shellable.




We complete this paper by a result showing that the problem on the number of 
``ears'' has a partial solution in a rather general geometric situation.

Let $S = \{p_i\}\subset {\Bbb R}^d$ be a set consisting of $n$ points
and let $t$ be a triangulation of $S$. The triangulation $t$ is {\it
regular\/} if there exists a strictly convex piecewise-linear
function, defined on the convex hull of $S$ and linear on each
simplex of $t$ (see~\cite{7, 16}).

We easily see that the Delaunay triangulation of the set $S$ is
regular. Indeed, For this purpose, we set the values $y_i =
\|p_i\|^2$ at vertices of the triangulation and continue the
function to the simplices by linearity. By the basic property of the
Delaunay triangulation \cite{12}, the function constructed is
convex.

It is known \cite{11} that any regular triangulation is
shellable. Then 
\begin{cor} Any regular triangulation of a generic convex simplicial polyhedron  
has at least one ``ear''.
\end{cor}

The fact that regularity implies shellabilty also shows that the triangulation constructed by Rudin is not
regular. For other examples of irregular triangulations, see, e.g.,
\cite{7}.

\medskip

Actually, using the method which is considered in \cite{11, 21, AGMT} it can be proved that any regular triangulation $t$ of a convex simplicial polytope has at least two ``ears''. Moreover,  $t$ and its dual triangulation have together at least $d+1$ ears. A proof of this theorem will be considered in our further paper.


\section{Conclusion}


One of the main aims of the present paper was to show that the
progress in discrete versions of the four-vertex theorem is not less
substantial than in the smooth case. These theorems have a long
history, which started with the Cauchy lemma. Furthermore, in the discrete case, there are higher-dimensional
generalizations. Possibly, they will help to formulate and prove
similar theorems for the smooth case.

\end{document}